\documentclass[11pt]{amsart}
\usepackage[T1]{fontenc}
\usepackage{lmodern,microtype}

\makeatletter
\def\@seccntformat#1{%
  \protect\textup{\protect\@secnumfont
    \ifnum\pdfstrcmp{subsection}{#1}=0 \bfseries\fi
    \csname the#1\endcsname
    \protect\@secnumpunct
  }%
}  
\makeatother

\let\oldtocsection=\tocsection

\let\oldtocsubsection=\tocsubsection

\let\oldtocsubsubsection=\tocsubsubsection

\renewcommand{\tocsection}[2]{\hspace{0em}\oldtocsection{#1}{#2}}
\renewcommand{\tocsubsection}[2]{\hspace{1em}\oldtocsubsection{#1}{#2}}
\renewcommand{\tocsubsubsection}[2]{\hspace{2em}\oldtocsubsubsection{#1}{#2}}

\usepackage{mathtools}

\usepackage{amsmath, amsthm}
\usepackage{geometry}
\usepackage{color}                
\usepackage{graphicx}
\usepackage{tikz}
\usepackage{tikz-cd}
\usetikzlibrary{matrix,arrows,arrows.meta,decorations.pathmorphing,decorations.markings,calc}
\usepackage{amssymb}

\usepackage{scalerel}

\usepackage[hidelinks,hypertexnames=false]{hyperref}
\let\textlstroke\l
\let\textLstroke\L

\newcommand{\ssub}{\subset\joinrel\subset}

\newcommand{\be}{\begin{equation}}
\newcommand{\bea}{\begin{eqnarray}}
\newcommand{\eea}{\end{eqnarray}} \newcommand{\ee}{\end{equation}}

\def\ba{\begin{eqnarray}}
\def\ea{\end{eqnarray}}

\def\v{\vskip .1in}

\def\al{\alpha}
\def\b{\beta}

\def\d{\delta}
\def\e{\varepsilon}
\def\g{\gamma}
\def\l{\lambda}
\def\m{\mu}
\def\n{\nu}
\def\o{\omega}
\def\f{\varphi}
\def\r{\rho}

\def\t{\theta}
\def\z{\zeta}
\def\k{\kappa}

\def\D{\Delta}
\def\O{\Omega}
\def\T{\Theta}

\def\cK{{\mathcal K}}

\def\Q{{\mathbb Q}}
\def\R{{\mathbb R}}
\def\C{{\mathbb C}}
\def\P{{\mathbb P}}

\def\diam{{\rm diam}}

\def\PSH{{\rm PSH}}

\def\Ric{{\rm Ric}}
\def\reg{{\rm reg}}

\def\sing{{\rm sing}}

\def\vol{{\rm vol}}

\def\x{{\bf x}}
\def\y{{\bf y}}

\def\tr{{\rm tr}}

\def\ti\tilde
\def\u{\underline}

\def\pl{\partial}

\def\i{\infty}
\def\I{\int}
\def\p{\prod}
\def\s{\sum}

\def\ddb{\partial\bar\partial}
\def\sub{\subseteq}
\def\ra{\rightarrow}
\def\hra{\hookrightarrow}

\def\lra{\longrightarrow}

\def\L{\Lambda}

\def\ti{\tilde}

\def\[{{\bf [}}
\def\]{{\bf ]}}

\def\pl{\partial}
\def\sq{{{\sqrt{{\scalebox{0.75}[1.0]{\( - 1\)}}}}}\hskip .01in}

\newtheorem{theorem}{Theorem}
\newtheorem{corollary}{Corollary}[section]
\newtheorem{lemma}{Lemma}

\newtheorem{proposition}{Proposition}

\newcommand{\ddc}{\sq\,\ddb}
\newcommand{\euc}{\mathrm{euc}}
\newcommand{\FS}{\mathrm{FS}}
\allowdisplaybreaks[2]
\title{Geometric H\"older estimates for complex Monge-Amp\`ere equations}
\author{Bin Guo}
\address{Department of Mathematics \& Computer Science, Rutgers University, Newark, NJ 07102, USA}
\email{bguo@rutgers.edu}
\author{Jian Song}
\address{Department of Mathematics, Rutgers University, Piscataway, NJ 08854, USA}
\email{jiansong@math.rutgers.edu}
\author{Jacob Sturm}
\address{Department of Mathematics \& Computer Science, Rutgers University, Newark, NJ 07102, USA}
\email{sturm@newark.rutgers.edu}
\thanks{Research supported in part by the National Science Foundation under grants DMS-2203607 and DMS-2505575.}
\date{}
\hypersetup{
  pdftitle={Geometric Holder estimates for complex Monge-Ampere equations},
  pdfauthor={Bin Guo, Jian Song, Jacob Sturm}
}
\begin{document}

\begin{abstract}
{This paper is the continuation of our earlier work on H\"older estimates for K\"ahler potentials on compact normal K\"ahler spaces. We establish a geometric counterpart for analytic H\"older regularity of Kolodziej. If a singular  metric $\o_\phi$ has H\"older continuous potentials with respect to a smooth background metric, then its distance function is H\"older continuous with respect to a smooth background distance. The result holds on normal K\"ahler spaces and yields compactness of the metric completion as well as its identification with the underlying complex space if $\o_\phi$ is a K\"ahler current. Under this positivity assumption, our estimates establish the H\"older equivalence between the intrinsic canonical K\"ahler metrics and the extrinsic smooth metrics on K\"ahler--RCD spaces, particularly on smoothable K\"ahler--Einstein spaces.}
\end{abstract}
\maketitle

\section{Introduction}

Let $(M,\t)$ be a compact K\"ahler manifold equipped with a smooth K\"ahler metric $\t$. We consider the following complex Monge--Amp\`ere equation
$$
(\t+\ddc\f)^n=f\t^n,\qquad \sup_M\f=0,
$$
where $f\geq0$ and $\I_M f\t^n=\I_M\t^n$.

\v
Yau's solution to the Calabi conjecture [8] guarantees the existence and uniqueness of the solution when $f>0$ is smooth. Ko\textlstroke{}odziej's work [9,10] shows that an $L^p$ density with $p>1$ forces boundedness and then H\"older continuity of the solution. The $L^\i$ estimate extends to big semipositive classes [11], in particular to compact normal K\"ahler varieties. However, the corresponding H\"older estimate remains elusive in the genuinely degenerate setting. The recent work [2] overcomes this difficulty for smoothable projective varieties by means of quantitative Bergman approximations. It proves uniform H\"older estimates under a Ricci lower bound and a diameter upper bound, with applications to smoothable singular K\"ahler--Einstein spaces and to degenerate Monge--Amp\`ere equations.

\v
Let $X$ be a compact connected normal K\"ahler space. By a smooth K\"ahler metric $\t$ on $X$ we mean that $\t$ is locally the restriction of a smooth ambient K\"ahler metric via a local holomorphic embedding of $X$. Its restriction to $X^{\reg}$ determines a length metric $d_\t$. 
We define $C^\al(X)$ using Euclidean distances in local holomorphic
embeddings: a function $f$ belongs to $C^\al(X)$ if, locally,
$$
|f(x)-f(y)|\le C\,|\iota(x)-\iota(y)|^\al,
$$
where $\iota:U\hookrightarrow\C^N$ is such an embedding.
On compact $X$, different finite choices of embeddings give equivalent
H\"older norms.

\v
Our aim is to show that analytic H\"older regularity implies metric H\"older regularity. Starting from H\"older continuity of $\f$, we would like to prove H\"older continuity of the distance induced by
$$
\o=\t+\ddc\f
$$
relative to $d_\t$. On a smooth K\"ahler manifold, this  follows from the work of Y.~Li [4]. His proof uses the volume-doubling property of balls for the background
metric $\t$, together with a scale-invariant Poincar\'e inequality.
\v
On singular spaces, however, the required scale-invariant $(1,2)$-Poincar\'e inequality can fail for the intrinsic metric induced by a smooth ambient K\"ahler form. For example, this occurs on $\{(x,y,z)\in\C^3:xy=z^k\}$, $k\geq3$, with the restricted Euclidean metric; see [12].

\v
We bypass the difficulties at singular points by pushing the metric to a smooth base via a finite projection. After deriving the distance estimate downstairs, we lift it back while keeping track of the monodromy.

\v
Here is a typical application.

\begin{theorem}
Let $(X,\o_{\mathrm{KE}})$ be a klt K\"ahler--Einstein variety, let $\o_0$ be a smooth background metric, and let $d_{\mathrm{KE}}$ and $d_0$ be the corresponding distance functions. Assume $X$ is smoothable. Then there exists $C>0$ and $\al>0$ such that
$$
\frac1C\cdot d_0\le d_{\mathrm{KE}}\le C\cdot d_0^\al.
$$
\end{theorem}

Here we use the polarized projective $\Q{}$-Gorenstein smoothability convention which can be found in [2, Definitions 1--2].

\v
In the case of smooth metrics on smooth manifolds, the fact that H\"older continuity of
potentials implies H\"older equivalence of distance functions is the content of the following
theorem of Li [4].

\begin{theorem}
If $\o_0$ is a K\"ahler metric on a compact complex manifold $X$ , and if $\f\in\PSH(X,\o_0)\cap C^\i(X)$, {$\o_\f>0$, and $0<\al\le1$,} then the distance function $d_{\o_\f}$ for
$\o_\f=\o_0+\ddc\f$ is H\"older continuous with respect to $d_0$, the distance function for $\o_0$ with H\"older exponent $\al/2$. In other words,
\be\tag{1.1}
d_{\o_\f}(x,y)\le C d_0(x,y)^{\al/2}
\ee
for all $x,y\in X$. Here $C$ depends on $(X,\o_0)$ and $\|\f\|_{C^\al(\o_0)}$.
\end{theorem}

On concentric coordinate balls $U_0\ssub U_1$, the same estimate holds for endpoints in $U_0$ and paths in $U_1$, with $C$ also depending on these balls.

\v
The proof of Theorem 2 in [4] makes use of the scale invariant Poincar\'e inequality, {which can fail in the singular case}. In Theorem 3 below, we introduce
a different approach which allows us to treat singular metrics and also provides refined
geometric information, even in the smooth setting.

\v
Theorem 2 says that $d_\f(p,q)\le A d_0(p,q)^{\al/2}$. In other words, for every $p,q\in X$, if $\g$ is the $\o_0$ geodesic joining $p$ to $q$, {there} is another (non-unique) curve $\eta$ joining $p$ to $q$ such that
\be\tag{1.2}
\ell_\f(\eta)\le C\ell_0(\g)^{\al/2}.
\ee
One may naturally ask for the relationship between the curves $\eta$ and $\g$. For example, can
we choose $\eta=\g$? In the following, we show that this can be done, up to an $\e$ error.

\begin{theorem}
Let $(X,\o_0)$ be a compact K\"ahler manifold and $\o=\o_0+\ddc\f$ a K\"ahler metric satisfying
$$
|\f(p)-\f(q)|\le Kd_0(p,q)^\al
$$
for all $p,q\in X$, {where $0<\al\le1$.} Then there exists $C>0$, depending on $(X,\o_0)$ and $K,\al$, with the following property. For every $\e\in(0,1/4]$ and for every $p,q\in X$ there are curves $\g$ and $\eta$ joining $p$ to $q$ such that
\begin{enumerate}
\item $d_0(p,q)\le\ell_0(\g)\le Cd_0(p,q)$,
\item $d_0(\g(t),\eta(t))<\e$ for all $t\in[0,1]$
\end{enumerate}
such that
$$
\ell_\f(\eta)\le\frac C\e\cdot\ell_0(\g)^{\al/2}
$$
and hence
\be\tag{1.3}
d_\f(p,q)\le\ell_\f(\eta)\le\frac C\e\cdot d_0(p,q)^{\al/2}.
\ee
\end{theorem}
\noindent\textit{Remark:} Taking $\e=1/4$ we recover Theorem 2.

\v
We next show how the technique {used} in the proof can be used to prove the following.

\begin{theorem}\label{thm:projective}
Let $W\sub\P^N$ be a connected normal projective variety, and let
$\f\in\PSH(W,\o_0)\cap C^\al(W)\cap C^\i(W^{\reg})$, where $0<\al\le1$.
Here $\o_0:=\o_{\FS}|_W$ and we set $\o_\f=\o_0+\ddc\f>0$.
Let $d_0,d_{\o_\f}$ be the distance functions on $W^{\reg}$ associated to
$\o_0$ and $\o_\f$, respectively. Fix a finite cover by projection
neighborhoods as in Section~\ref{sec:distance}, using only $(W,\o_0)$,
and let $D$ be the largest degree of these projections. Then there is
$C>0$, depending only on the fixed background data, $\al$ and
$\|\f\|_{C^\al(W)}$, such that for all $x,x'\in W^{\reg}$,
$$
 d_{\o_\f}(x,x')\le C d_0(x,x')^{\al/(2D)}.
$$
\end{theorem}

\begin{corollary}\label{cor:completion}
Under the hypotheses of Theorem~\ref{thm:projective}, the metric completion
$\widehat W$ of $(W^{\reg},d_{\o_\f})$ is compact. If
$\o_\f\ge c\o_0$ for some $c>0$, the identity on $W^{\reg}$ extends to a
canonical homeomorphism $W\simeq\widehat W$, and the completed distances satisfy
$$
\sqrt c\,d_0(x,x')\le d_{\o_\f}(x,x')
\le C d_0(x,x')^{\al/(2D)},\qquad x,x'\in W.
$$
\end{corollary}

Here
$$
d_{\o_\f}(x',x)=\inf\left\{\I_\g|\g'(t)|_{\o_\f}\,dt:
\g:[0,1]\ra W^{\reg},\ \g(0)=x,\ \g(1)=x'\right\}.
$$
{Here the curves are piecewise smooth, and $C^\al(W)$ denotes local ambient H\"older continuity, with $0<\al\le1$.}

\v
\noindent\textit{Remark:} We state the theorem here for normal varieties, but the proof works as well in the case of compact connected normal K\"ahler spaces; see Subsection~\ref{subsec:normal-kahler}.

\v
The proof is local and gives exponent $\al/(2d)$ on a projection
neighborhood of degree $d$. The global choice $\al/(2D)$ is not intended
to be optimal; $D$ depends on the fixed cover and is independent of $\f$.
Corollary~\ref{cor:completion} separates compactness of the metric completion
from its identification with the complex space under strict positivity.

\v
We now place the result in the K\"ahler--RCD setting. Let $X$ be a compact connected normal K\"ahler variety with klt singularities, and let $\t$ be a smooth K\"ahler metric obtained from local or global holomorphic embeddings. For example, on a projective variety it may be the restriction of a Fubini--Study metric. Let $\O$ be an adapted volume measure: on a chart where a nonvanishing section $\eta$ generates $mK_X$, its restriction to $X^{\reg}$ is, up to a smooth positive factor,
$$
(\eta\wedge\bar\eta)^{1/m},
$$
with the usual positivity convention. For $p>1$, let $\cK_p(X,\t,\O)$ be the class of closed positive currents $\o=\t+\ddc\f$ with $\f\in\PSH(X,\t)\cap L^\i(X)$, smooth and K\"ahler on $X^{\reg}$, and satisfying
$$
\frac{\o^n}{\O}\in L^p(X,\O).
$$
Write $(\widehat X,d,\m)$ for the metric completion of $(X^{\reg},d_\o)$ equipped with the canonical Borel extension of $\o^n$. We call this a K\"ahler--RCD space when it is an $\operatorname{RCD}(K,2n)$ space for some $K\in\R{}$. Compactness and uniform Sobolev estimates in this setting are treated in [13,14].

\v
In the geometric K\"ahler--RCD settings of [15,16,6,17,18], the completion is canonically homeomorphic to the analytic variety. This identification is a theorem in those works, rather than a formal consequence of the RCD axioms. In our general statement we first compare distances on $X^{\reg}$ and obtain the identification under the lower metric bound.

\v
The final application, Theorem~\ref{thm:rcd}, uses the analytic H\"older estimate supplied by the companion work [7], in preparation. All preceding proofs, including the smoothable K\"ahler--Einstein and uniform projective applications, do not use this input. If $\o\in\cK_p(X,\t,\O)$ and $(\widehat X,d,\m)$ is noncollapsed $\operatorname{RCD}(-\l,2n)$, combining that analytic input with our distance estimate gives constants $C>0$ and $\d\in(0,1]$ such that
$$
d_\o(x,y)\le C d_\t(x,y)^\d,\qquad x,y\in X^{\reg}.
$$
If also $\o\ge c_0\t$, the completion is identified with $X$ and
$$
c_0^{1/2}d_\t(x,y)\le d(x,y)\le C d_\t(x,y)^\d,\qquad x,y\in X.
$$
The smoothable K\"ahler--Einstein case is Theorem 1. Theorem~\ref{thm:uniform} gives the comparison uniformly for polarized families with Ricci curvature bounded below and diameter bounded above.

\v
The paper is organized as follows. Section 2 proves the smooth distance estimate of Theorem 3. Section 3 gives the algebraic root estimate. Section 4 proves Theorem 4 by finite projection and monodromy, proves Corollary~\ref{cor:completion}, and extends these results to normal K\"ahler spaces. Section 5 treats the uniform projective application first and the K\"ahler--RCD application last.

\v
\noindent\textit{Acknowledgments.} The main result was communicated to Valentino Tosatti in May 2026. The authors thank him for insightful discussions.

\section{Proof of Theorem 3}
\begin{proof}
As in the statement of the theorem, we shall use $C>0$ to denote a constant
depending on $(X,\o_0)$ and $K,\al$ (but independent of $\e$).
We allow the constant $C$ to change value from line to line to avoid introducing
new constants $C_1,C_2,\ldots$ etc.

\v
First assume $p\ne q$ and $d_0(p,q)$ is below a fixed background coordinate scale, so all segments and cutoff supports below remain in one chart.

\v
We work in local coordinates so that $C^{-1}\o_{\mathrm{euc}}\leq\o_0
\leq C\o_{\mathrm{euc}}$. Here $\o_{\mathrm{euc}}$ is the euclidean
metric on $\C^n$. Note that
\be\tag{2.1}
\o_\f
=\o_0+\sq\ddb\f
=\o_{\mathrm{euc}}+\sq\ddb(\f-\psi)
\ee
where $\psi$ is a local potential for
{$\o_{\mathrm{euc}}-\o_0$}.
Since $\|\psi\|_{C^\al}\leq C$, {enlarging $K$ if necessary,} we will write $\f$ instead of
$\f-\psi$ so that in local coordinates,
$\o_\f=\o_{\mathrm{euc}}+\sq\ddb\f$.

\v
We may assume $p=(0,0,\ldots,0)$ and $q=(q,0,\ldots,0)\in\C^n$.
Then $\g$, the straight line segment joining $p$ to $q$, is a
$\o_{\mathrm{euc}}$ geodesic and
$C^{-1}d_0(p,q)\leq |q|=\ell_{\mathrm{euc}}(\g)\leq Cd_0(p,q)$.

\v
Let $r=\e|q|=\e d$. Let $m=q/2$ be the midpoint of the line
segment joining $p$ to $q$.

\v
For $x\in B_r(m)$, the euclidean ball of radius $r$ centered at $m$, we define
$\eta_{1,x}:[0,1]\ra\C^n$ by the formula $\eta_{1,x}(t)=tx$.
This is the euclidean segment joining $p$ to $x$. Similarly we let $\eta_{2,x}$
to be the euclidean segment joining $x$ to $q$ and we define
$\eta_x=\eta_{1,x}\cup\eta_{2,x}$. We wish to show that for an appropriately
chosen value of $x$ that $\eta$ and $\g$ satisfy (1.3).

\v
Parametrizing $\g$ and $\eta_x$ affinely on each half of $[0,1]$ gives $|\eta_x(t)-\g(t)|\le|x-m|<r$. Shrinking the fixed coordinate scale ensures $d_0(\eta_x(t),\g(t))<\e$.

\v
We have
\begin{align*}
\ell_\f(\eta_{1,x})
&=\I_0^1\bigl(g_{i\bar j}x^i\bar x^j\bigr)^{1/2}\,dt\\
&\leq\I_0^1
\bigl(g^{\mathrm{euc}}_{i\bar j}x^i\bar x^j\cdot
\tr_{g_{\mathrm{euc}}}(g)\bigr)^{1/2}\,dt\\
&=|x|\I_0^1\bigl[n+(\D_{g_{\mathrm{euc}}}\f)(tx)\bigr]^{1/2}\,dt.
\end{align*}
Let $\r(x)\in C_c^\i(B_{2r}(m){)}$ be the usual cutoff function which
equals one on $B_r(m)${, with $0\leq\r\leq1$ and
$|\D_{\mathrm{euc}}\r|\leq Cr^{-2}$}, let $dx$ denote Euclidean measure,
and let $V=\vol_{\mathrm{euc}}(B_{2r}(m))$, and
$d=|q|=d_{\mathrm{euc}}$.
{Also put $V_r=\vol_{\mathrm{euc}}(B_r(m))$, so that
$V=2^{2n}V_r$.}
Then
\begin{align}
\I_{B_r(m)}\ell_\f(\eta_{1,x})\,dx
&\leq\I_{B_{2r}(m)}\ell_\f(\eta_{1,x})\r(x)^2\,dx\notag\\
&\leq d\I_{B_{2r}(m)}\I_0^1
\bigl[(n+(\D_{\mathrm{euc}}\f)(tx))\r(x)\bigr]^{1/2}\,dt\,dx\notag\\
&\leq d\cdot V^{1/2}\I_0^1
\left[\I_{B_{2r}}(n+(\D_{\mathrm{euc}}\f)(tx))\r(x)\,dx\right]^{1/2}\,dt\notag\\
&=d\cdot V^{1/2}\I_0^1
\left[\I_{B_{2r}}\left(n+\frac1{t^2}\D_{\mathrm{euc}}
(\f(tx)-\f(0))\right)\r(x)\,dx\right]^{1/2}\,dt\notag\\
&=d\cdot V^{1/2}\I_0^1
\left[\I_{B_{2r}}\left(n\r(x)+\frac1{t^2}\cdot
(\f(tx)-\f(0))\D\r(x)\right)\,dx\right]^{1/2}\,dt\tag{2.2}\\
&\leq Cd\cdot V^{1/2}\I_0^1
\left[nV+\I_{B_{2r}}\frac1{t^2}|\f(tx)-\f(0)|\frac1{r^2}\,dx\right]^{1/2}\,dt\notag\\
&\leq Cd\cdot V^{1/2}\I_0^1
\left(V^{1/2}+\left[\I_{B_{2r}}\frac1{t^2}
|\f(tx)-\f(0)|\frac1{r^2}\,dx\right]^{1/2}\right)\,dt\notag\\
&\leq Cd\cdot V+Cd\cdot V^{1/2}\cdot K^{1/2}\cdot\frac1r
\I_0^1\frac1t\left[\I_{B_{2r}}|tx|^\al\,dx\right]^{1/2}\,dt\notag\\
&\leq Cd\cdot V+Cd\cdot V^{1/2}\cdot\frac1r
\I_0^1t^{\al/2-1}\left[\I_{B_{2r}}|x|^\al\,dx\right]^{1/2}\,dt\notag\\
&=Cd\cdot V+Cd\cdot V^{1/2}\cdot\frac1r
\left[\I_0^1t^{\al/2-1}\,dt\right]
\left[\I_{B_{2r}}|x|^\al\,dx\right]^{1/2}\notag\\
&\leq Cd\cdot V+\frac C\e\cdot V^{1/2}
\cdot\left[\frac1\al\right]\cdot[d^\al\cdot V]^{1/2}\notag\\
&=Cd\cdot V+\frac1\e\cdot Cd^{\al/2}\cdot V
\leq\frac1\e\cdot Cd^{\al/2}\cdot V.\notag
\end{align}
The same inequality holds if we replace $\eta_{1,x}$ with $\eta_{2,x}$.
Adding the two we obtain
$$
\frac1{V_r}\I_{B_r(m)}\ell_\f(\eta_x)\,dx
\leq\frac C\e\cdot d^{\al/2}
\leq\frac C\e\cdot d_0(p,q)^{\al/2}.
$$
Thus there exists $W_r\sub B_r(m)$, a set of positive measure, such that
for all $x_0\in W_r$
\be\tag{2.3}
d_\f(p,q)\leq\ell_\f(\eta_{x_0})
\leq{\frac1{V_r}}
\I_{B_r(m)}\ell_\f(\eta_x)\,dx
\leq\frac C\e\cdot d_0(p,q)^{\al/2}.
\ee

\v
For arbitrary $p,q$, subdivide a minimizing $\o_0$-geodesic into a uniformly bounded number of sufficiently short subarcs and concatenate the local constructions. The length estimate follows by concavity, and the parameterwise $\e$-closeness is preserved.

\end{proof}

\section{Growth of Algebraic Functions}

In order to prove Theorem 4 we shall need some preliminary results on growth of algebraic
functions.

\v
Let $\pi:X\ra Z$ be a covering space and $\g:[0,1]\ra Z$ a continuous curve. Let
$z=\g(0)$, $z'=\g(1)$, and $x\in X_z:=\pi^{-1}(z)$.
{The} ``lift of $\g$ to $x$'' is the unique curve
$\l:[0,1]\ra X$ such that $\l(0)=x$ and $\pi\circ\l=\g$.
In particular, $\l(1)=x'$ for some $x'\in\pi^{-1}(z')$.

\v
Let $B_r(p)=\{z\in\C{}:|z-p|<r\}$, let $d\geq1$ and let
$a_1,\ldots,a_d$ be holomorphic functions on $B_R(0)^n\sub\C^n$,
where $R=4d(4d+4)^{d-1}$.
{In this section, for $z\in\C^n$, $|z|$ denotes the supremum norm
$\max_\n|z_\n|$.}
Assume $\sup_{B_R(0)^n}|a_j|<1$ for $1\leq j\leq d$ and let
$$
f(z,x)=x^d+\s_{j=1}^d a_j(z)x^{d-j}.
$$
Note that by Fujiwara's bound, for a fixed $z\in B_R(0)^n$, if $f(z,x)=0$
then $|x|<2$. Let $B\sub B_R(0)^n$ be the subvariety consisting of
those points $z\in B_R(0)^n$ defined by the property
$|\{x\in\C{}:f(z,x)=0\}|<d$. Here the roots of the polynomial
$f(z,x)=0$ are counted without multiplicity. Let
$X'=\{(z,x)\in B_1(0)^n\times B_2(0):f(z,x)=0\}$ and let
$\pi:X'\ra B_1(0)^n$ be the map $(z,x)\mapsto z$.
Let $Z=B_1(0)^n\setminus B$ and $X=\pi^{-1}(Z)$. Then
$\pi:X\ra Z$ is a covering map of degree $d$.

\begin{lemma}\label{lem:root}
Let $a_1,\ldots,a_d$ be holomorphic functions as above, and define $f(z,x)$
and $B$ as above. Let $z',z\in Z$ and assume the curve
$\g:[0,1]\ra Z$ has the properties
\begin{enumerate}
\item $\g(0)=z$, $\g(1)=z'$.
\item $\g(t)\notin B$ for all $t\in[0,1]$.
\item $|\g(t)-z|\leq|z'-z|$ for all $t\in[0,1]$.
\end{enumerate}

{Let $(z,x)\in\pi^{-1}(z)$, let
$\l(t)=(\g(t),\xi(t))$ be the lift of $\g$ starting at $(z,x)$,
and let $x'=\xi(1)$.}
Then
\be
{|\xi(1)-\xi(0)|}
=|x'-x|\leq(2d+1)|z'-z|^{1/d}.
\tag{3.1}\label{eq:root-displacement}
\ee
\end{lemma}

\begin{proof}
Fix $z$, $z'$ and $x$ as above and write
$$
f(z,x)=\p_{j=1}^d(x-x_j)
$$
with $x=x_1$. We order the $x_j$ such that
$$
0\leq|x_2-x_1|\leq|x_3-x_1|\leq\cdots\leq|x_d-x_1|<\i.
$$
Let $r=|z'-z|$ and $\al=r^{1/d}$.
{We may assume $r>0$, the case $r=0$ being immediate.}
Then, by the pigeonhole principle, there exists $1\leq p\leq2d+1$ and
$1\leq j\leq d$ such that
$$
|x_j-x_1|\leq(p-1)\al<p\al<(p+1)\al\leq|x_{j+1}-x_1|,
$$
where we define $|x_{d+1}-x_1|:=\i$.

\v
We claim that for each $t\in[0,1]$, the polynomials
$f(z,x),f(\g(t),x)\in\C{}[x]$ have the same number of roots in
the ball $B_{p\al}(x_1)$. To see this, we use Rouch\'e's theorem:
if $|x-x_1|=p\al$ then
$$
|f(z,x)|=\p_{j=1}^d|x-x_j|\geq\al^d=r.
$$
On the other hand, if $|x-x_1|=p\al$ then
$|x|{\leq}p\al+|x_1|{<}(2d+1)2+2=4d+4$ so
$$
\begin{aligned}
|f(\g(t),x)-f(z,x)|
&=\left|\s_{j=1}^d(a_j(\g(t))-a_j(z))
{x^{d-j}}\right|\\
&{\leq}
d(4d+4)^{d-1}\max_{1\leq j\leq d}|a_j(\g(t))-a_j(z)|.
\end{aligned}
$$

\v
Fix $t\in(0,1]$ with $\g(t)\neq\g(0)=z$ and
{let} $\C^1\sub\C^n$ be the {affine} complex line containing
$\g(0)$ and $\g(t)$
{, using affine coordinates in which $z=0$, $\g(t)=s:=|\g(t)-z|$,
and the direction vector has supremum norm one}.
{The circle $\pl B_{R-1}(0)$ then lies in the coefficient polydisc,
and $s\leq r<2$, $R\geq4$.} Then
$$
\begin{aligned}
|a_j(\g(t))-a_j(z)|
&={\Bigl|}\frac1{2\pi i}
\I_{{\pl B_{R-1}(0)}}
\frac{a_j(\z)(\g(t)-z)}{(\z-z)(\z-\g(t))}\,d\z
{\Bigr|}\\
&
{<\frac{s}{R-1-s}\leq\frac{r}{R-3}}
{\leq}\frac{4r}{R}.
\end{aligned}
$$
Thus
$|f(\g(t),{x})-f(z,x)|
<d(4d+4)^{d-1}\frac{4r}{R}=r$.

\v
Rouch\'e's theorem now implies that
{$\xi(t)\in B_{pr^{1/d}}(x)$ for all $t$, by continuity of $\xi$.}
In particular, setting $t=1$, we obtain
$$
{|\xi(1)-\xi(0)|}
=|x'-x|\leq pr^{1/d}\leq(2d+1)r^{1/d}
=(2d+1)|z'-z|^{1/d}.
$$
\end{proof}

{The same proof applies to any continuous choice of a root, even when roots coincide.}

\v
\noindent\textit{Remark:} If we choose $R$ to be any number $R>1$ and assume
$|a_j(z)|\leq A$ on ${B_R(0)^n}$, then a scaling argument
shows that (3.1) still holds with the constant $(2d+1)$ replaced by a
constant depending on $d$, $A$ and $R$.

\section{Bi-H\"older continuity of the distance function}
\label{sec:distance}

In this section we prove Theorem 4. First we recall the definition of a  {plurisubharmonic}
function on a singular variety. We say $\f\in\PSH(X,\o_0)$ {if, locally, $h+\f=\Phi|_X$, where $\o_0=\ddc h$,} where $\Phi$
is a PSH function on an open set in $\C^M$ and $\Phi|_X$ is the restriction of $\Phi$ to a local embedding
of $X$ in $\C^M$.

\medskip
\noindent\textit{Proof of Theorem 4.}
The problem is local so if we let $\C^N\sub\P^N$ be an affine open set, we
can, for the purpose of proving the theorem, replace $W$ by {a relatively compact neighborhood in $W\cap\C^N$}.
By the Noether normalization theorem, if $\C^n\sub\C^N$ is a generic vector subspace then
$\pi:{W\cap\C^N}\ra\C^n$ is a finite map in the sense of algebraic geometry. In particular, $\pi$ is proper
and if $d$ is the degree of $\pi$, then the fibers of $\pi$ contain exactly $d$ points counted with
multiplicity {in the sense of local mapping degree}. For example, if $X=\{x_1,x_2\in\C{}:x_1x_2=1\}$ then we can take $\C^1$ to be any
line other than $x_1=0$ or $x_2=0$. A point in $\C^N$ will be labelled $(z_1,\ldots,z_n,w_1,\ldots,w_{N-n})$
so $\pi$ is the projection onto the first $n$ coordinates.

\v
Let $Y=B_R(0)^n$ and let $B\sub Y$ be a subvariety containing the image of the singular set as
well as the branch locus. Let $Z=Y\setminus B$ and $X=\pi^{-1}(Z)\sub\C^N$. Then $\pi|_X:X\ra Z$ is a
covering space whose degree we denote by $d$.

\v
Use nested base balls $Y_1\ssub Y_2\ssub Y\ssub Y_{\mathrm{out}}$, with $\pi^{-1}(\overline Y)$ relatively compact in the affine chart. Distances downstairs use paths in $Y\setminus B$, and (4.6) concerns endpoints in $Y_2\setminus B$. The corresponding construction for normal K\"ahler spaces is given in Subsection~\ref{subsec:normal-kahler}.

\v
Recall that $\o_0=\o_{\FS}|_W$ and $\o_\f=\o_0+\ddc\f$. Let
$$
\eta_0=\pi_*\o_0\quad\text{and}\quad
\eta_\psi=\pi_*\o_\f=\eta_0+\ddc\psi
$$
and
$$
\o_\psi=\pi^*\eta_\psi=
{\pi^*\eta_0}+\ddc(\psi\circ\pi).
$$
Then $\eta_0,\eta_\psi$ are positive $(1,1)$ currents on $Y$ which are smooth K\"ahler metrics on $Z$, and
$\o_\psi$ is a positive $(1,1)$ current which is smooth on $X$. Moreover, for $z\in Z$ we have
\be\tag{4.1}
\psi(z)=\s_{j=1}^d\f(x_j)\quad\text{where}\quad\pi^{-1}(z)=\{x_1,\ldots,x_d\}.
\ee
If $z\in Y$, then (4.1) still holds if we count the points in $\pi^{-1}(z)$ with multiplicity. In
particular, $\eta_0$ and $\eta_\psi$ are positive $(1,1)$ currents on $Y=B_R(0)^n$ with continuous local
potentials whose restrictions to $Z\sub Y=B_R(0)^n$ are smooth K\"ahler metrics. In the next
lemma we estimate the H\"older continuity of these potentials.

\v
If $\o_0=\ddc h$, these continuous potentials are $\psi_0=\pi_*h$ and
$F=\pi_*(h+\f)$; continuity follows by conservation of local degree.
On $Z$, $F$ is a sum of psh functions composed with the local inverse
branches of $\pi$. Its continuous extension across $B$ is therefore psh,
and $\ddc F=\eta_\psi$. The inverse-branch formula
$\eta_\psi=\s_j s_j^*\o_\f$ gives $\o_\psi\ge\o_\f$ on $X$.

\begin{lemma}\label{lem:trace-holder}
There is a constant $C>0$ such that for every $z,z'\in Y$,
\be\tag{4.2}
|\psi(z')-\psi(z)|+|F(z')-F(z)|\le C|z'-z|^{\al/d}.
\ee
The constant depends only on the fixed projection data, $\al$, and the
ambient $C^\al$ bound of $\f$.
\end{lemma}

\begin{proof}
We use the notation $\ell_{\eta_\psi}(\g)$ to denote the length of a smooth curve {$\g$ on $Z$} with
respect to the $\eta_\psi$ metric. If $\g:[0,1]\ra Z$ is smooth then {its $d$ lifts satisfy}
\be\tag{4.3}
\max_{1\le j\le d}\ell_{\o_\f}(\l_j)
\le\ell_{\eta_\psi}(\g)
\le\s_{j=1}^d\ell_{\o_\f}(\l_j)
\le\sqrt d\,\ell_{\eta_\psi}(\g).
\ee
Here the $\{\l_1,\ldots,\l_d\}$ are the lifts of $\g$ to the points $x_1,\ldots,x_d$. In particular, $\ell_{\o_\psi}(\l_p)$ is
independent of $p$ and $\o_\psi{\ge}\o_\f$.

\v
Now let $\g(t)$ be a curve satisfying the hypotheses of Lemma 1. Let $z=\g(0)$ and $z'=\g(1)$.
Order the set $\pi^{-1}(z')=\{x'_1,\ldots,x'_d\}$ so that $x_j,x'_j$ are the endpoints of $\l_j$.

\v
For $r=|z'-z|>0$ below the fixed local scale, choose a generic midpoint
$a$ within $r/4$ of $(z+z')/2$ in the norm of Section 3, and use the
broken segment $[z,a]\cup[a,z']$. It stays within $|\z-z|\le r$;
almost every such midpoint gives segments avoiding $B$. The coordinate
polynomials
$$
P_\n(z,T)=\p_{j=1}^d(T-w_\n(x_j))
$$
have bounded holomorphic coefficients extending across $B$. Thus the
scaled form of Lemma~\ref{lem:root} applies coordinatewise to the lifts
of this one path, including when a coordinate polynomial has repeated roots.
Equivalence of norms on the fixed ambient space gives
$|x'_j-x_j|_{\euc}\le C r^{1/d}$ for every $j$.

\v
Use the ambient H\"older assumption directly. Since $h$ is smooth on a
neighborhood of the compact set under consideration, both $\f$ and
$h+\f$ have bounded ambient $C^\al$ seminorms. Hence
\be\tag{4.4}
\begin{aligned}
|\psi(z')-\psi(z)|
&=\left|\s_{j=1}^d\f(x'_j)-\s_{j=1}^d\f(x_j)\right|\\
&\le\s_{j=1}^d|\f(x'_j)-\f(x_j)|
\le C\s_{j=1}^d|x'_j-x_j|_{\euc}^{\al},
\end{aligned}
\ee
and the same calculation with $h+\f$ gives
\be\tag{4.5}
|\psi(z')-\psi(z)|+|F(z')-F(z)|\le C|z'-z|^{\al/d}.
\ee
Continuity extends this inequality across $B$. Bounded oscillation
handles pairs whose separation exceeds the fixed local scale.

\end{proof}

We will also use the following local intrinsic--extrinsic comparison,
a consequence of the subanalytic H\"older comparison theorem of
\textLstroke{}ojasiewicz [5]. It identifies the background completion
and permits comparison with ambient endpoint distances; it is not needed
in the exponent calculation (4.5).

\begin{theorem}\label{thm:background}
Let $Y\sub\C^N$ be {a normal complex analytic space. Around each point choose sufficiently small neighborhoods $U'\ssub U$ for which $U\cap Y^{\reg}$ is connected}.
Let $\o$ {be} the euclidean metric on $\C^N$ and $\o_0$ its restriction to $Y$.
Then there exists $C>0$ and $\b\in(0,1)$ such that for all $x,x'\in{U'\cap Y^{\reg}}$ we have
$$
d_0(x,x'):=d_{\o_0}(x,x')\le C|x'-x|^\b,
$$
{where the distance on the left is measured by curves in $U\cap Y^{\reg}$.}
\end{theorem}

On relatively compact charts, a smooth ambient K\"ahler metric is
uniformly equivalent to the Euclidean metric, so Theorem~\ref{thm:background}
applies to its intrinsic distance as well. The reverse local bound
$|x-x'|\le C d_0(x,x')$ follows by comparing the length of a curve with
its ambient displacement. Normality ensures local irreducibility and
connectedness of the regular locus in sufficiently small neighborhoods.
Consequently, regular points approaching the same point of the analytic
space form a single Cauchy class for $d_0$; distinct points give distinct
classes. These local comparisons identify the background completion of
a compact normal space with the space itself. In particular, that
completion is compact and has finite diameter.

\begin{lemma}\label{lem:base-distance}
There exists $C>0$ such that
\be\tag{4.6}
d_{\eta_\psi}(z',z)\le C|z'-z|^\tau,
\qquad \tau=\frac{\al}{2d},
\ee
for all $z,z'\in Y_2\setminus B$, with paths measured in $Y\setminus B$.
\end{lemma}

\begin{proof}
The proof is the same averaging argument as in Theorem 3. Put
$\k=\al/d$. By Lemma~\ref{lem:trace-holder}, $F=\psi_0+\psi$ is a
$C^\k$ psh function and $\eta_\psi=\ddc F$. In a small Euclidean ball,
write
$$
\eta_\psi=\o_{\euc}+\ddc(F-|z|^2).
$$
The potential $F-|z|^2$ has a bounded $C^\k$ seminorm. Thus (2.1)
holds with the smooth Euclidean background, even though $\eta_\psi$
itself need not be smooth along $B$.

\v
For endpoints at Euclidean distance $s>0$, average over broken segments
through a ball of radius $s/8$ about their midpoint. After translating
the first endpoint to $0$, the function $x\mapsto F(tx)$ is bounded
and psh for each $0<t\le1$. Integration by parts in (2.2) is valid in the
sense of distributions. The smooth trace density off $B$ is bounded
by the full positive trace measure of $\ddc F$, so the same calculation
gives
$$
\frac1{\vol B_{s/8}}\I_{B_{s/8}}
\ell_{\eta_\psi}(\eta_x)\,dx
\le C\bigl(s+s^{\k/2}\bigr)
\le C s^{\k/2}
$$
when $s\le1$. The integral in the parameter is
$\I_0^1t^{\k/2-1}\,dt=2/\k<\i$.
Almost every broken segment avoids the analytic set $B$, and choosing
one with length bounded by the average proves the estimate below a
fixed scale. Subdivision inside $Y$ proves it for all endpoints in
$Y_2\setminus B$, with a larger constant. This gives
$\tau=\k/2=\al/(2d)$.
\end{proof}

The inequality (4.6) is nearly what we want; we just want to lift this estimate back up to
an estimate on $X$, i.e.\ we want to use (4.6) to show that
\be\tag{4.7}
d_{\o_\f}\le Cd_{\mathrm{Euc}}^{\d_{\mathrm{ext}}}
\ee
for some $\d_{\mathrm{ext}}>0$. We will first obtain the intrinsic exponent $\tau=\al/(2d)$; the comparison with ambient distance is a further consequence. It suffices to show
\be\tag{4.8}
d_{\o_\psi}(x_0,x_1)\le Cd_{\mathrm{Euc}}(x_0,x_1)^{\d_{\mathrm{ext}}}
\ee
where $\o_\psi=\pi^*\pi_*\o_\f$ (since $\o_\psi{\ge}\o_\f$). Here $d_{\mathrm{Euc}}$ is the euclidean metric on {the chosen relatively compact projection neighborhood} and
by abuse of notation, we use the notation $\psi$ to denote $\psi\circ\pi$.

\medskip
\noindent\textit{Proof of Theorem 4.}

Take $x_0\in\pi^{-1}(Y_1)\cap X$ and $d_{\o_0}(x_0,x_1),\e$ small, so that the short curves below project into $Y_2$. Curves in $W^{\reg}$ with endpoints outside $\pi^{-1}(B)$ may be slightly perturbed to avoid this set, keeping their endpoints fixed and increasing their $\o_0$-length arbitrarily little.

\v
Let $x_0,x_1\in X$ and $x_t=\l(t)$ a curve joining $x_0$ to $x_1$ such that
$$
\ell_{\o_0}(\l)\leq d_{\o_0}(x_0,x_1)+\e.
$$
We call $\l$ an $(\o_0,\e)$-geodesic.
{Choose $\l$ piecewise smooth.}

\v
Let $z_j=\pi(x_j)$, and $z(t)=\pi\circ\l(t)$. Then $z(t)$ joins $z_0$ to $z_1$ and clearly $z(t)$ lifts to $\l(t)$.
{Write $\g=z$ and $\pi^{-1}(z_0)=\{x_0^1,\ldots,x_0^d\}$, with $x_0^1=x_0$.}
For $s<t$, let $\l[s,t]$ (resp. $z[s,t]$) be $\l$ (resp. $z$) restricted to the interval $[s,t]\cap[0,1]$.
{For $s>t$ use the reversed subarc.}

\v
For $1\leq k\leq d$ we define $S_k{\sub}[0,1]$ as follows: $t\in S_k$ if there exist a sequence of $(\eta_\psi,1/j)$-geodesics $\xi_j[0,t]:[0,t]\ra Z$ joining $z_0$ to $z_t=\pi(x_t)$ with the following property: the closed curve
$$
\qquad
{\xi_j[0,t]^{-1}\circ\g[0,t]}
$$
lifts to a curve in $X$ joining $x_0^1$ to $x_0^k$.

\v
Equivalently, the lift of $\xi_j[0,t]$ from $x_0^k$ ends at $\l(t)$.

\v
The $S_k$ cover $[0,1]$: lift reversed approximate geodesics from $\l(t)$ and select a subsequence ending at one fixed fiber point. Each $S_k$ is closed, since appending $\g[s,t]$ to one of the paths defining $s\in S_k$ adds an error tending to zero as $s\ra t$; hence every nonempty $S_k$ contains its supremum. Finally, $S_1$ contains an initial interval, by minimizing geodesics in an evenly covered normal neighborhood of $z_0$.

\v
Let $m$ be a positive integer with $m\leq d$. We define real numbers $0=t_0<t_1<\cdots<t_m=1$ as follows.
Let $t_0=0$ and $t_1=\sup S_1$. If $t_1=1$ then we set $m=1$. If $t_1<1$ then, after possibly reordering {the unused fiber points $x_0^2,\ldots,x_0^d$}, we may assume that $S_2\cap[t_1,1]$ has $t_1$ as an accumulation point. Let $t_2=\sup S_2$. If $t_2=1$ then let $m=2$. Otherwise, we continue this process until $t_m=1$.

\v
The finite cover supplies an accumulation set from the right at each $t_k<1$. It is unused because all earlier suprema are at most $t_k$; thus the suprema strictly increase and reach $1$ after at most $d$ steps.

\v
Fix $1\leq k\leq m-1$. Choose $\e_{k,j}\ra0$ such that $t_k+\e_{k,j}\in S_{k+1}$ and choose
{$\xi_{k,j}[0,t_k]$ and $\xi_{k+1,j}[0,t_k+\e_{k,j}]$}
to be $(\eta_\psi,1/j)$ approximate geodesics satisfying:
\begin{enumerate}
\item They join $z_0$ to $z(t_k)$ and $z_0$ to $z(t_k+\e_{k,j})$ respectively.
\item

{$\xi_{k,j}[0,t_k]^{-1}\circ\g[0,t_k]$}
lifts to a curve in $X$ joining $x_0^1$ to $x_0^k$.
\item

{$\xi_{k+1,j}[0,t_k+\e_{k,j}]^{-1}\circ\g[0,t_k+\e_{k,j}]$}
lifts to a curve in $X$ joining $x_0^1$ to $x_0^{k+1}$ (possible since $t_k+\e_{k,j}\in S_{k+1}$).
\end{enumerate}

Choose these paths independently at their indicated endpoints, with $0<\e_{k,j}<t_{k+1}-t_k$, and choose $\xi_{m,j}[0,1]$ using $1\in S_m$.

\v
Set $\e_{m,j}=0$ and
$$
\d_j=\s_{k=1}^{m-1}
 \ell_{\eta_\psi}\bigl(\g[t_k,t_k+\e_{k,j}]\bigr)\lra0.
$$

\v
In Figure 1, the red dots represent $z(t_0),z(t_1),z(t_2),z(t_3),\ldots,z(t_m)$, the blue dots represent
{$z(t_1+\e_{1,j}),z(t_2+\e_{2,j}),\ldots$}
and the black dots are the elements of $B$ (and thus don't intersect the dotted curves).
The choice of
{the paths and $\e_{k,j}$}
implies that the boundaries of the yellow regions in Figure 2 lift to closed curves on $X$.

\v
Put, with product factors ordered from $k=m-1$ on the left to $k=1$ on the right,
$$
\begin{aligned}
\t_j:={}&\left(\p_{k=m-1,\ldots,1}
 \g[t_{k+1},t_{k+1}+\e_{k+1,j}]
 \circ\xi_{k+1,j}[0,t_{k+1}]
 \circ\xi_{k+1,j}[0,t_k+\e_{k,j}]^{-1}\right)\\
&\circ\g[t_1,t_1+\e_{1,j}]\circ\xi_{1,j}[0,t_1],
\end{aligned}
$$
and let $\T_j$ be its lift from $x_0$ (the product is empty if $m=1$). After the first path and connector the lift is at $\l(t_1+\e_{1,j})$; each product factor then lifts through $x_0^{k+1}$ to $\l(t_{k+1}+\e_{k+1,j})$. The final endpoint is therefore $\l(1)=x_1$.

\begin{figure}[ht]
\centering
\resizebox{\linewidth}{!}{%
\begin{tikzpicture}[
  x=0.83cm,y=0.83cm,
  line cap=round,line join=round,
  font=\small,
  outgoing/.style={postaction={decorate},decoration={markings,
    mark=at position .64 with {\arrow[green!55!black]{Triangle[length=3mm,width=2.5mm]}}}},
  returning/.style={postaction={decorate},decoration={markings,
    mark=at position .70 with {\arrowreversed[green!55!black]{Triangle[length=3mm,width=2.5mm]}}}},
  arc/.style={blue,line width=.9pt,dash pattern=on 4pt off 2.5pt}
]
\path[use as bounding box] (-.45,-1.15) rectangle (14.1,5.1);
\fill[white] (-.45,-1.15) rectangle (14.1,5.1);
\draw[black,line width=1pt] (0,0)--(14,0);

\draw[arc,outgoing] (0,0) .. controls (1.5,1.08) and (3.5,1.36) .. (3.3,0);
\draw[arc,returning] (0,0) .. controls (1.65,1.38) and (3.62,1.63) .. (3.85,0);
\draw[arc,outgoing] (0,0) .. controls (2.50,2.35) and (5.65,3.08) .. (5,0);
\draw[arc,returning] (0,0) .. controls (2.85,2.76) and (5.80,3.30) .. (5.55,0);
\draw[arc,outgoing] (0,0) .. controls (3.90,4.45) and (8.05,5.90) .. (8.05,0);
\draw[arc,returning] (0,0) .. controls (4.0,4.82) and (8.2,6.12) .. (8.60,0);
\draw[red,line width=.9pt,dash pattern=on 4pt off 2.5pt,outgoing]
  (0,0) .. controls (4.8,6.3) and (8.85,5.85) .. (13.2,0);

\foreach \x/\y in {3.02/.90,3.62/1.84,5.62/1.98,6.68/1.92,6.28/1.04,9.32/2.08,10.45/2.08}
  \fill[black] (\x,\y) circle[radius=1.9pt];
\foreach \x in {0,3.3,5,8.05,13.2}
  \fill[red] (\x,0) circle[radius=2.5pt];
\foreach \x in {3.85,5.55,8.60}
  \fill[blue] (\x,0) circle[radius=2.5pt];

\node[red,below=5pt] at (0,0) {$z_0$};
\node[red,below=5pt] at (3.3,0) {$z(t_1)$};
\node[red,below=5pt] at (5,0) {$z(t_2)$};
\node[red,below=5pt] at (8.05,0) {$z(t_3)$};
\node[red,below=5pt] at (13.2,0) {$z(t_m)$};
\node[blue] at (3.55,-.86) {$z(t_1+\e_{1,j})$};
\node[blue] at (6.02,-.86) {$z(t_2+\e_{2,j})$};
\node[blue] at (9.02,-.86) {$z(t_3+\e_{3,j})$};
\node[blue] at (6.35,2.73) {$\xi_{3,j}[0,t_3]$};
\node[blue,font=\footnotesize] at (9.60,2.50) {$\xi_{4,j}[0,t_3+\e_{3,j}]$};
\end{tikzpicture}
}
\caption{}
\end{figure}
\begin{figure}[ht]
\centering
\resizebox{\linewidth}{!}{%
\begin{tikzpicture}[x=.70cm,y=.70cm,
  figure path/.style={draw=figureblue,line width=.9pt,
    dash pattern=on 3pt off 2pt},
  outward/.style={postaction={decorate},decoration={markings,
    mark=at position .64 with {\arrow[figuregreen]{Triangle[length=2.6mm,width=2.5mm]}}}},
  inward/.style={postaction={decorate},decoration={markings,
    mark=at position .66 with {\arrow[figuregreen]{Triangle[reversed,length=2.6mm,width=2.5mm]}}}},
  every node/.style={font=\small,inner sep=1pt}]
  \definecolor{figureblue}{RGB}{12,68,217}
  \definecolor{figurered}{RGB}{239,38,29}
  \definecolor{figuregreen}{RGB}{0,132,36}
  \definecolor{figureyellow}{RGB}{255,250,180}

  \fill[white] (-.55,-1.23) rectangle (17.65,5.70);

  \fill[figureyellow]
    (0,0) .. controls (1.6,.77) and (3.5,1.38) .. (3.5,0)
    -- cycle;
  \fill[figureyellow]
    (0,0) .. controls (3.6,2.50) and (6.9,3.20) .. (6.4,0)
    -- (4.1,0)
    .. controls (4.1,1.65) and (1.7,1.20) .. (0,0) -- cycle;
  \fill[figureyellow]
    (0,0) .. controls (6.3,6.03) and (10.6,5.98) .. (10.6,0)
    -- (7,0)
    .. controls (7,3.65) and (3.5,2.97) .. (0,0) -- cycle;
  \fill[figureyellow]
    (0,0) .. controls (6.35,7.80) and (12.1,6.60) .. (17,0)
    -- (11.2,0)
    .. controls (10.8,6.68) and (6.4,6.25) .. (0,0) -- cycle;

  \draw[figure path,outward]
    (0,0) .. controls (1.6,.77) and (3.5,1.38) .. (3.5,0);
  \draw[figure path,inward]
    (0,0) .. controls (1.7,1.20) and (4.1,1.65) .. (4.1,0);
  \draw[figure path,outward]
    (0,0) .. controls (3.6,2.50) and (6.9,3.20) .. (6.4,0);
  \draw[figure path,inward]
    (0,0) .. controls (3.5,2.97) and (7,3.65) .. (7,0);
  \draw[figure path,outward]
    (0,0) .. controls (6.3,6.03) and (10.6,5.98) .. (10.6,0);
  \draw[figure path,inward]
    (0,0) .. controls (6.4,6.25) and (10.8,6.68) .. (11.2,0);
  \draw[figure path,outward]
    (0,0) .. controls (6.35,7.80) and (12.1,6.60) .. (17,0);

  \draw[black,line width=1.1pt] (0,0)--(17.6,0);
  \foreach \x in {0,3.5,6.4,10.6,17}
    \fill[figurered] (\x,0) circle[radius=2.4pt];
  \foreach \x in {4.1,7,11.2}
    \fill[figureblue] (\x,0) circle[radius=2.4pt];

  \node[anchor=north] at (0,-.23) {$z_0$};
  \node[anchor=north] at (3.5,-.23) {$z(t_1)$};
  \node[anchor=north] at (6.4,-.23) {$z(t_2)$};
  \node[anchor=north] at (10.6,-.23) {$z(t_3)$};
  \node[anchor=north] at (17,-.23) {$z(t_m)$};
  \node[anchor=north] at (4.1,-.75) {$z(t_1+\e_{1,j})$};
  \node[anchor=north] at (7,-.75) {$z(t_2+\e_{2,j})$};
  \node[anchor=north] at (11.2,-.75) {$z(t_3+\e_{3,j})$};
  \pgfresetboundingbox
  \path[use as bounding box] (-.55,-1.23) rectangle (17.65,5.70);
\end{tikzpicture}
}
\caption{}
\end{figure}

Now (4.6) implies
$$
\begin{aligned}
d_{\eta_\psi}(z_0,z_t)
&\leq C|z_t-z_0|^\tau
 \leq C\ell_{\o_0}(\l[0,t])^\tau\\
&\leq C\bigl(d_{\o_0}(x_0,x_1)+\e\bigr)^\tau
 =:L_\e.
\end{aligned}
$$
By (4.3), the $2m-1$ approximate geodesics and the connectors give
$$
d_{\o_\f}(x_0,x_1)
\leq\ell_{\o_\f}(\T_j)
\leq(2m-1)(L_\e+1/j)+\d_j.
$$
First let $j\ra\i$ with $\l$ fixed, then let $\e\ra0$. Since $m\leq d$,
$$
d_{\o_\f}(x_0,x_1)\leq (2d-1)C\,d_{\o_0}(x_0,x_1)^\tau.
$$
Here $\tau=\al/(2d)$, so this is the required local intrinsic estimate.
The same estimate holds for $d_{\o_\psi}$ because
$\o_\psi=\pi^*\eta_\psi$. Theorem~\ref{thm:background} also gives
(4.8) and (4.7), with $\d_{\mathrm{ext}}=\b\tau$; this additional
conversion does not change the exponent of the intrinsic estimate.

\v
Smoothness of $\o_0,\o_\f$ near regular endpoints extends the intrinsic
estimate across the branch preimage by approximation. Choose finitely
many of the smaller projection neighborhoods to cover $W$, and let
$d_i$ be their degrees. Fix this cover using only $(W,\o_0)$ and put
$$
D=\max_i d_i,\qquad \g=\frac{\al}{2D}.
$$
Below a fixed scale $s_0\le1$, the local estimates therefore give
$d_{\o_\f}(x,y)\le C d_0(x,y)^\g$ whenever the first endpoint lies
in one of the smaller neighborhoods and the endpoints are sufficiently
close. The nested neighborhoods ensure that a sufficiently short
$\o_0$-almost-minimizing curve remains in the corresponding larger
neighborhood.

\v
By Theorem~\ref{thm:background}, the $d_0$-completion is the compact
space $W$ and has finite diameter. Subdivide a $d_0$-almost-minimizing
curve from $x$ to $y$ into at most $N_0$ sufficiently short subarcs,
where $N_0$ depends only on the fixed background data. If their lengths
are $s_j$, the local estimates and concavity give
$$
d_{\o_\f}(x,y)\le C\s_j s_j^\g
\le C N_0^{1-\g}\Bigl(\s_j s_j\Bigr)^\g
\le C N_0^{1-\g}\bigl(d_0(x,y)+\e\bigr)^\g.
$$
Letting $\e\ra0$ proves Theorem~\ref{thm:projective}.\hfill$\square$

\medskip
\noindent\textit{Proof of Corollary~\ref{cor:completion}.}
The upper distance estimate extends the identity on $W^{\reg}$ to a
continuous map from its $d_0$-completion $W$ to its $d_{\o_\f}$-completion
$\widehat W$. The image is compact and contains the dense regular locus,
so the map is surjective and $\widehat W$ is compact. If
$\o_\f\ge c\o_0$, comparison of lengths gives
$d_{\o_\f}\ge\sqrt c\,d_0$ on $W^{\reg}$. The inverse identity then
extends continuously in the other direction. The two maps are inverse
on the dense regular locus, hence everywhere. The two distance bounds
extend by continuity.\hfill$\square$

\subsection{Extension to normal K\"ahler spaces}
\label{subsec:normal-kahler}

The projective hypothesis in Theorem~\ref{thm:projective} supplies finite
projections and smooth local potentials for the background metric.
Both constructions are available locally on a normal K\"ahler space.
We give the details, keeping the distance and monodromy arguments above.

\begin{proposition}\label{prop:normal-kahler}
Let $(X,\t)$ be a compact connected normal K\"ahler space. Fix a finite
cover by local finite projection neighborhoods, as constructed below,
using only $(X,\t)$, and let their degrees be $d_1,\ldots,d_s$.
Put $D=\max_i d_i$. Suppose that $0<\al\le1$,
$$
\f\in\PSH(X,\t)\cap C^\al(X)\cap C^2(X^{\reg}),
\qquad \o=\t+\ddc\f,
$$
and that $\o$ is a smooth K\"ahler metric on $X^{\reg}$.
Here $C^\al$ has the ambient meaning specified in the introduction.
Then
$$
d_\o(x,y)\le C d_\t(x,y)^{\al/(2D)},\qquad x,y\in X^{\reg},
$$
where $C$ depends only on the fixed background data, $\al$, and
$\|\f\|_{C^\al(X)}$. The metric completion is compact. If
$\o\ge c\t$ for some $c>0$, it is canonically homeomorphic to $X$,
and the completed distances satisfy
$$
\sqrt c\,d_\t(x,y)\le d_\o(x,y)
\le C d_\t(x,y)^{\al/(2D)},\qquad x,y\in X.
$$
\end{proposition}

\begin{proof}
Fix $a\in X$ and a local holomorphic embedding near $a$ in $\C^M$.
Choose this embedding small enough that $\t=\ddc h$, where $h$ is the
restriction of a smooth strictly psh ambient potential. The local
parametrization theorem [19, Chapter II, Theorem 4.19] gives, after
shrinking the neighborhood, a finite proper holomorphic projection
$$
\pi:U\lra Y_{\mathrm{out}}\sub\C^n
$$
of degree $d$. Choose nested base balls
$Y_1\ssub Y_2\ssub Y\ssub Y_{\mathrm{out}}$ whose smallest inverse
image contains $a$. Properness makes the inverse images of the smaller
closed balls compact in $U$. Compactness of $X$ supplies finitely many
such projections whose inverse images of $Y_1$ cover $X$. This system
is fixed independently of $\f$.

\v
Let $B$ be a proper analytic subset of the base containing the branch
locus and the image of $U^{\sing}$. Over $Z=Y_{\mathrm{out}}\setminus B$,
the map is an unramified covering of degree $d$. For each remaining
ambient coordinate $w_\n$, form
$$
P_\n(z,T)=\p_{j=1}^d\bigl(T-w_\n(x_j)\bigr),
\qquad \pi^{-1}(z)=\{x_1,\ldots,x_d\},\quad z\in Z.
$$
Its coefficients are single-valued holomorphic functions on $Z$.
They are locally bounded, because the relevant inverse images are
relatively compact. The removable singularity theorem on the smooth
base extends the coefficients holomorphically across $B$; Cauchy
estimates control them on smaller base balls. Multiplicities at
exceptional fibers are local mapping degrees, whose sum is $d$.
This construction does not require flatness of $\pi$.

\v
Put $u=h+\f$. On the smaller neighborhood, $h$ is ambient Lipschitz,
and $u$ has an ambient $C^\al$ bound controlled by the stated data.
Define
$$
F(z)=\s_{x\in\pi^{-1}(z)}u(x),
$$
counting local mapping degrees as multiplicities. Conservation of
degree and continuity of $u$ give continuity of $F$. Off $B$, $F$ is
the sum of the psh functions $u\circ s_j$, where $s_j$ are local inverse
branches. Its continuous extension across $B$ is therefore psh.

\v
For nearby $z,z'\in Y\setminus B$, choose a short broken segment
avoiding $B$, as in Lemma~\ref{lem:trace-holder}. Applying
Lemma~\ref{lem:root} to the polynomials $P_\n$ pairs the endpoints of
its lifts so that
$$
|x'_j-x_j|_{\euc}\le C|z'-z|^{1/d}.
$$
The continuous-root version is relevant here, since coordinate
polynomials can have repeated roots even over an unramified fiber.
Consequently,
$$
|F(z')-F(z)|
\le \s_{j=1}^d|u(x'_j)-u(x_j)|
\le C|z'-z|^{\al/d}.
$$
Continuity extends the inequality across $B$, and bounded oscillation
handles pairs beyond the fixed local scale.

\v
On $Y\setminus B$ the form
$$
\eta=\ddc F=\s_{j=1}^d s_j^*\o
$$
is a smooth K\"ahler metric, with $\pi^*\eta\ge\o$. The averaging
argument of Lemma~\ref{lem:base-distance}, applied to the psh
$C^{\al/d}$ potential $F$, gives
$$
d_\eta(z,z')\le C|z-z'|^{\al/(2d)},
\qquad z,z'\in Y_2\setminus B,
$$
where the distance uses curves in $Y\setminus B$.

\v
A sufficiently short $\t$-almost-minimizing curve beginning over $Y_1$
stays over $Y_2$. Indeed, $\t$ is uniformly comparable to the restricted
ambient Euclidean metric on the compact set $\pi^{-1}(\overline{Y_2})$,
and leaving the larger neighborhood from the smaller one requires a
fixed positive length. Thus this assertion applies to curves approximating
the global distance $d_\t$. With regular endpoints outside $\pi^{-1}(B)$,
these curves can be slightly perturbed to avoid $\pi^{-1}(B)$, with fixed
endpoints and arbitrarily small increase of $\t$-length. Their projected
displacements are bounded by a constant times their $\t$-lengths.
The monodromy construction in the proof of Theorem~\ref{thm:projective}
therefore applies unchanged. It uses at most $2d-1$ short base curves
and gives
$$
d_\o(x,y)\le C d_\t(x,y)^{\al/(2d)}
$$
for sufficiently close endpoints in the smaller neighborhood.
Approximation extends this estimate across the branch preimage at
regular endpoints.

\v
Theorem~\ref{thm:background} and local comparison of $\t$ with the
ambient Euclidean metric identify the $d_\t$-completion with $X$.
It is compact and has finite diameter. Choose a uniform positive
scale on which the local estimates just proved apply, and decrease it
to at most one. Since $d_i\le D$, each local estimate then holds with
exponent $\al/(2D)$. The same subdivision and concavity argument used
at the end of the proof of Theorem~\ref{thm:projective} proves the global
inequality, with constants depending only on the stated data.

\v
Finally, the proof of Corollary~\ref{cor:completion} applies verbatim:
the upper estimate gives a continuous surjection from $X$ onto the
$d_\o$-completion, and the lower metric bound, when assumed, gives a
continuous inverse. Both distance inequalities extend to the completion.
\end{proof}

\subsection{The smoothable K\"ahler--Einstein application}

We recall the two analytic results used here. Under the polarized
projective $\Q{}$-Gorenstein smoothability convention specified after
Theorem 1, [2, Theorem 2] gives ambient H\"older continuity of the
potential of the singular K\"ahler--Einstein metric relative to a smooth
background form in its class. The strict positivity theorem
[6, Theorem 1.1] gives a lower bound by a positive multiple of a smooth
background K\"ahler metric. After a common rescaling, choose
$\t=\o_{\FS}|_X\in[\o_{\mathrm{KE}}]$. These results provide
$\al\in(0,1]$ and $c>0$ such that
$$
\o_{\mathrm{KE}}=\t+\ddc\f,\qquad
\f\in\PSH(X,\t)\cap C^\al(X),\qquad
\o_{\mathrm{KE}}\ge c\t.
$$
The potential and metric are smooth on $X^{\reg}$.

\medskip
\noindent\textit{Proof of Theorem 1.}
Apply Theorem~\ref{thm:projective} and Corollary~\ref{cor:completion}
to the preceding potential and metric. On the compact space $X$, the
smooth background metrics $\t$ and $\o_0$ are uniformly equivalent;
their intrinsic distances are therefore comparable. This gives the
two asserted inequalities and identifies the completed distance with
a distance on $X$.\hfill$\square$

\section{Applications}
\label{sec:applications}

\subsection{Uniform estimates for polarized families}

We first state the analytic estimate that will be used in the proof.
For fixed $n\ge1$ and $D>0$, consider polarized compact K\"ahler
manifolds $(X,\o,L,h)$ with
$$
\dim_{\C{}}X=n,\qquad \Ric(\o)\ge-\o,\qquad
\diam_\o X\le D,\qquad -\ddc\log h=\o.
$$
The analytic H\"older estimate and the projective embedding results
in [2, Theorem 1 and Section 2] give integers $k,N$ and constants
$A>0$, $\al\in(0,1]$, depending only on $n,D$, with the following
properties. For the inner product defined by $h^k$ and $(k\o)^n$,
an orthonormal basis of $H^0(X,L^k)$ defines an embedding
$T_k:X\hra\P^{N_X}$ with $N_X\le N$. If
$$
\t_X=k^{-1}T_k^*\o_{\FS},
$$
there is a normalized potential $\f_X$, with $\sup_X\f_X=0$, such that
$$
\o=\t_X+\ddc\f_X,\qquad
|\f_X(x)-\f_X(y)|
\le A\,d_{(\P^{N_X},k^{-1}\o_{\FS})}(T_kx,T_ky)^\al.
$$
The distance on the right is the ambient projective distance; this
distinction is important when applying the finite projection estimate.
Since that ambient space has uniformly bounded diameter, the same
estimate and normalization give a uniform bound on $\|\f_X\|_{L^\i}$.
Here and below $D$ in this subsection denotes the diameter bound.

\v
We combine this analytic input with the geometric argument of Section 4.

\begin{theorem}[Uniform projective comparison]\label{thm:uniform}
For every $n\geq1$ and $D>0$, there are integers $k,N$ and constants
$C>0$, $\d\in(0,1]$, depending only on $n,D$, with the following
property. Let $(X,\o,L,h)$ be a polarized compact K\"ahler manifold with
$$
\dim_{\C{}}X=n,\qquad
\Ric{}(\o)\geq-\o,\qquad
\diam_\o X\leq D,\qquad -\ddc\log h=\o.
$$
The Kodaira map $T_k:X\lra\P^{N_X}$ defined by an
$L^2$-orthonormal basis of $H^0(X,L^k)$, using $h^k$ and $(k\o)^n$,
is an embedding with $N_X\leq N$. For $\t_X=k^{-1}T_k^*\o_{\FS}$,
\be\tag{5.1}
C^{-1}d_{\t_X}(p,q)\leq d_\o(p,q)
\leq C d_{\t_X}(p,q)^\d,\qquad p,q\in X.
\ee
\end{theorem}

\begin{proof}
All constants below depend only on $n,D$. Integrality of $c_1(L)$
gives a positive lower bound for volume, and Bishop comparison with
the Ricci and diameter bounds gives an upper bound. Choose $k,N,A,\al$
as in the analytic input above. The partial $C^0$ estimate bounds the
Bergman density below, while the Bochner--Moser estimates bound the
sections and their first derivatives. In the formula for
$k^{-1}T_k^*\o_{\FS}$, these estimates bound its coefficients relative
to $\o$, giving $\t_X\le C\o$. The analytic input gives
\be\tag{5.2}
\o=\t_X+\ddc\f_X,\qquad
|\f_X(x)-\f_X(y)|\leq
A\,d_{(\P^{N_X},k^{-1}\o_{\FS})}(T_kx,T_ky)^\al.
\ee

\v
The degree of $T_k(X)$ is bounded by an integer $E=E(n,D)$.
View these images in $\P^N$. For each cycle in the compact union
of the Chow spaces of $n$-cycles of degrees at most $E$, choose a
projective linear projection to $\P^n$ whose center misses its
support. This projection is finite on the support, and the center
remains a positive distance from all nearby supports. A finite cover
therefore supplies finitely many projections with a uniform separation
from their centers. Choose finitely many affine base charts with nested
balls $Y_1\ssub Y_2\ssub Y\ssub Y_{\mathrm{out}}$ whose smallest
balls cover $\P^n$.

\v
On the inverse images of these balls the affine coordinates are
uniformly bounded, and the projection degree satisfies $d\leq E$.
The coordinate characteristic polynomials consequently have uniformly
bounded coefficients, and Cauchy estimates bound their derivatives on
$Y$. Thus Lemma 1 applies uniformly. This argument uses only the
supports of the limiting cycles; repeated roots are allowed.

\v
Let $h_0$ be the local Fubini--Study potential for $\t_X$.
It has a uniform ambient Lipschitz bound on these charts. Applying
Lemma 1 directly to the ambient estimate (5.2) shows that
$F=\pi_*(h_0+\f_X)$ has a uniform $C^{\al/d}$ bound, after
subtracting a constant. The averaging argument of Section 4 gives
$$
d_{\pi_*\o}(z,z')\leq C|z-z'|^{\al/(2d)},
\qquad z,z'\in Y_2\setminus B.
$$
A sufficiently short $\t_X$-approximate geodesic beginning over
$Y_1$ projects into $Y_2$. The monodromy construction of Section 4,
with at most $2d-1$ short paths, therefore yields
$$
d_\o(x,y)\leq C d_{\t_X}(x,y)^{\al/(2E)}
$$
below a uniform scale. Smoothness extends this estimate across the
branch preimage. For pairs separated by more than that scale,
$d_\o\leq D$ gives the same estimate after increasing $C$.
Finally $\t_X\leq C\o$ gives the lower bound in (5.1).
\end{proof}

\subsection{K\"ahler--RCD spaces}
Let $X$ be a compact connected normal K\"ahler variety of dimension $n$ with klt
singularities, $\t$ a smooth K\"ahler metric, and $\O$ an adapted volume
measure. Locally, $\O$ is a smooth positive multiple of
$(\eta\wedge\bar\eta)^{1/m}$, with the usual positivity convention, where $\eta$
generates $mK_X$. For $p>1$, write $\cK_p(X,\t,\O)$ for the
currents $\o=\t+\ddc\f\ge0$ with bounded potential, smooth and
K\"ahler on $X^{\reg}$, and $\o^n/\O\in L^p(X,\O)$.

\v
This final application uses the following analytic input from the
companion work [7], which is in preparation. Suppose
$\o\in\cK_p(X,\t,\O)$ and its metric completion, with the canonical
extension of $\o^n$, is noncollapsed $\operatorname{RCD}(-\l,2n)$.
For the normalized potential $\sup_X\f=0$, the input asserts that
there are $A>0$ and $a\in(0,1]$ such that
$$
|\f(x)-\f(y)|\le A d_\t(x,y)^a,
\qquad x,y\in X^{\reg}.
$$
This intrinsic H\"older estimate extends $\f$ continuously to the
background completion $X$. We use precisely this conclusion of [7].
For fixed $(X,\t)$, the constants in the geometric conclusion below
are controlled by $A,a$ and the fixed background data. The result
is conditional on this analytic input; no preceding proof uses it.

\begin{theorem}\label{thm:rcd}
Let $\o\in\cK_p(X,\t,\O)$, and let
$(\widehat X,d,\m)$ be the metric completion of $(X^{\reg},d_\o)$,
with the canonical extension of $\o^n$. Suppose this is a noncollapsed
$\operatorname{RCD}(-\l,2n)$ space. Subject to the analytic input [7],
there exist $C>0$ and $\d\in(0,1]$ such that
\be\tag{5.3}
d_\o(x,y)\le C d_\t(x,y)^\d,
\qquad x,y\in X^{\reg}.
\ee
For fixed $(X,\t)$, these constants depend on the exponent $a$ and
bound $A$ supplied by the analytic input. The identity extends to a continuous surjection
$X\ra\widehat X$, and $\widehat X$ is compact. If $\o\ge c_0\t$
for some $c_0>0$, this map is a homeomorphism and
$$
c_0^{1/2}d_\t(x,y)\le d(x,y)\le C d_\t(x,y)^\d,
\qquad x,y\in X.
$$
\end{theorem}

\begin{proof}
Choose a fixed finite system of embedded neighborhoods for $X$.
By Theorem~\ref{thm:background}, after shrinking these neighborhoods,
there are $C_0>0$ and $\b\in(0,1]$ such that
$$
d_\t(x,y)\le C_0|x-y|^\b
$$
for nearby regular points in each smaller chart. Choose $\b$ to be
the minimum of the finitely many local exponents. The analytic input
therefore gives
$$
|\f(x)-\f(y)|\le A C_0^a |x-y|^{a\b}.
$$
Continuity extends the estimate to all points in the smaller charts.
Normalization and the finite background diameter also bound
$\|\f\|_{L^\i}$ in terms of $A,a$ and $(X,\t)$. Thus $\f$ has an
ambient $C^{a\b}$ bound with the stated dependence.

\v
On $X^{\reg}$, the current $\ddc\f=\o-\t$ is smooth;
local elliptic regularity therefore gives $\f\in C^\i(X^{\reg})$.
Fix the local projections of Proposition~\ref{prop:normal-kahler},
and denote their largest degree by $D_0$. That proposition gives (5.3)
with the explicit choice
$$
\d=\frac{a\b}{2D_0}.
$$
It also gives compactness, the continuous surjection from $X$ onto
$\widehat X$, and the claimed homeomorphism and lower estimate when
$\o\ge c_0\t$. If the analytic input is stated directly in the ambient
$C^\al$ sense, the same argument uses $\d=\al/(2D_0)$ without the
intrinsic--extrinsic conversion.
\end{proof}

\end{document}